\documentclass[12pt,reqno]{amsart}

\usepackage{amsmath,amssymb,amsthm}
\usepackage{amssymb}
\usepackage{amstext}
\usepackage{latexsym}
\usepackage{enumitem}
\usepackage[left=3.3cm, right=3.3cm, paperheight=792.0pt, footskip=.5in]{geometry}
\usepackage{hyperref}
\usepackage{url}
\usepackage{mathrsfs}

\usepackage{fancyhdr}
\newtheorem{theorem}{Theorem}
\newtheorem{lemma}{Lemma}

\newtheorem{corollary}{Corollary}

\providecommand{\Z}{{\mathbb Z}}
\providecommand{\N}{{\mathbb N}}
\providecommand{\G}{{\mathscr G}}
\providecommand{\R}{{\mathscr R}}
\providecommand{\E}{{\mathcal E}}

\hypersetup{
    pdftitle={Sums of dilates in ordered groups},
    pdfauthor={Alain Plagne and Salvatore Tringali},
    pdfmenubar=false,
    pdffitwindow=true,
    pdfstartview=FitH,
    colorlinks=true,
    linkcolor=red,
    citecolor=green,
    urlcolor=cyan
}

\begin{document}

\title{Sums of dilates in ordered groups}

\author{Alain Plagne}
\address{Centre de math\'ematiques Laurent Schwartz, \'Ecole polytechnique \\
91128 Palaiseau cedex, France}
\email{plagne@math.polytechnique.fr}
\urladdr{http://www.math.polytechnique.fr/~plagne/}

\author{Salvatore Tringali}
\address{Science Program, Texas A\&M University at Qatar, Education City \\ PO Box 23874 Doha, Qatar}
\email{salvo.tringali@gmail.com}
\urladdr{http://www.math.polytechnique.fr/~tringali/}


\thanks{This research was supported by the French ANR Project ``CAESAR'' No. ANR-12-BS01-0011.}

\begin{abstract}
We address the ``sums of dilates problem'' by looking for non-trivial lower bounds on sumsets of the form $k \cdot X + l \cdot X$, where $k$ and $l$ are non-zero integers and $X$ is a subset of a possibly non-abelian group $G$ (written additively). In particular, we investigate the extension of some results so far known only for the integers to the context of torsion-free or linearly orderable groups, either abelian or not.
\end{abstract}

\pagenumbering{arabic}

\maketitle

\section{Introduction}
\label{sec:intro_a041}

One of the main problems in additive combinatorics is to derive non-trivial bounds on the cardinality of  \textit{Minkowski sumsets} (usually called just sumsets), see \cite[Chapter 1]{Na} for the basic terminology and notation.
In this respect, the inequality:
\begin{equation}
\label{equ:basic}
|X+Y| \ge |X|+|Y|-1,
\end{equation}
where $X$ and $Y$ are non-empty subsets of the additive group of the integers, serves as a cornerstone and suggests at least two possible directions of research:
\begin{enumerate}[label={\rm (\roman{*})}]
\item extending \eqref{equ:basic} to broader contexts, e.g. more general groups or possibly non-cancellative semigroups (abelian or not), see \cite{TrMTK} and references therein;
\item improving \eqref{equ:basic} for special pairs of sets $(X,Y)$.
\end{enumerate}
Having this in mind, let $G$ be a fixed group, which needs not be abelian, but we write additively (unless a statement to the contrary is made). For $X$ a subset of $G$
and $k$ an integer, we define
$$
k \cdot X = \{ kx: x \in X\},
$$
and we call $k\cdot X$ a \textit{dilate} of $X$ (we use $X^{\cdot k}$ in place of $k \cdot X$ when $G$ is written multiplicatively).

The present paper is focused on some aspects of the following problem, which we refer to as the \textit{sums of dilates problem}:
Given integers $k$ and $l$, find the minimal possible cardinality of a sumset of the form $k\cdot X + l \cdot X$ in terms of the cardinality of $X$ or,
equivalently, find a best possible lower bound for $| k\cdot X + l \cdot X |$.
As far as we are aware, the study of this problem started about a decade ago, and until recent years it was mainly concerned with the case of $(\mathbb{Z}, +)$.

Specifically, early contributions to this subject date back at least to 2002, when Hamidoune and the first-named author obtained, as a by-product of an extension
of Freiman's $3k-3$ theorem, the following elementary estimate, valid for any non-empty set
$X \subseteq \mathbb{Z}$ and any integer $k$ with $|k| \ge 2$, see \cite{Hami00} and references therein:
\begin{equation}
\label{th:hamidoune_and_plagne_estimate_a041}
|X+k\cdot X| \ge 3|X|-2.
\end{equation}
This result was refined by Nathanson \cite{N}, who showed that
\begin{equation}
\label{eq:nathanson}
|k\cdot X + l\cdot X| \geq   \frac{7}{2}|X| -3
\end{equation}
for any pair of positive coprime integers $(k, l)$, at least one of which is $\ge 3$. Notice here that there is no loss of generality in assuming that $k$ and $l$ are coprime, since replacing $X$ with a set of the form $a \cdot X+b$,
where $a$ and $b$ are integers and $a \neq 0$, does not change the cardinality of the left- and right-hand sides of \eqref{eq:nathanson}.

The case $(k,l) = (1,3)$ was then completely settled in \cite{CSV} by Cilleruelo, Silva and Vinuesa, who established that
\begin{equation}
\label{equ:1+3_case_sharp}
|X+3\cdot X| \geq 4|X|-4,
\end{equation}
and determined the cases in which \eqref{equ:1+3_case_sharp} is an equality. In the same paper, the authors prove the inverse theorem that
$|X+2\cdot X| = 3|X| - 2$ holds if and only if $X$ is an arithmetic progression.

On another hand, the general problem of estimating the size of a sum of any finite number of dilates was considered in 2007 by Bukh \cite{BB}, who showed that
\begin{equation}
\label{equ:bukh_a041}
|k_1\cdot X + \cdots + k_n\cdot X| \geq (|k_1| + \cdots + |k_n|)|X| -  o(|X|)
\end{equation}
for any non-empty set $X \subseteq {\mathbb Z}$  and coprime integers $k_1, \ldots, k_n$. It has been since then conjectured that the bound \eqref{equ:bukh_a041} can be actually improved to the effect of replacing the term $o(|X|)$
with an absolute constant depending only on $k_1, \ldots, k_n$.

Together with applications of Bukh's results to the study of sum-product phenomena in finite fields \cite{Gara07}, the above  has motivated an increasing interest in the topic and led to a number of publications.
In particular, Cilleruelo, Hamidoune and Serra proved in \cite{CHS} that if $k$ is a (positive) prime and $|X| \ge 3(k-1)^3(k-2)!$ then
\begin{equation}
\label{csv_conjecture_a041}
|X+k\cdot X| \geq (1+k) |X| - \left\lceil \frac{k(k+2)}{4} \right\rceil,
\end{equation}
and they also determined all cases where \eqref{csv_conjecture_a041} holds as an equality. This in turn supports the conjecture, first suggested in \cite{CSV},
that \eqref{csv_conjecture_a041} is true for any positive integer $k$, provided that $|X|$ is sufficiently large. Likewise, Hamidoune and Ru\'e showed in \cite{HR} that
$|2\cdot X+k\cdot X| \ge (2+k)|X| - 4k^{k-1}$
if $k$ is again a prime, and in fact
$$
|2\cdot X + k\cdot X| \ge (2+k)|X| - k^2 - k + 2
$$
if $|X| > 8k^k$.

In cases where the previous bounds do not apply, weaker but nontrivial estimates can however be obtained.
For instance, Freiman, Herzog, Longobardi, Maj
and Stanchescu \cite{Freim13} have recently improved \eqref{equ:1+3_case_sharp} by establishing that
$$
|X+k\cdot X| \ge 4|X| - 4
$$
for any $k \ge 3$.
In the same paper, the authors prove various direct and inverse theorems on sums of dilates in the integers
which they use to obtain direct and inverse theorems for sumsets in certain submonoids of Baumslag-Solitar groups \cite{BS},
that is two-generator one-relator groups defined by a presentation of the form
$$
\langle a, b \mid b a^m b^{-1} = a^n \rangle,
$$
where $m$ and $n$ are non-zero integers (throughout, as is usual, these and other finitely presented groups are written multiplicatively).

Finally, even more recently, the case of $(\mathbb{Z}, +)$ was almost completely settled by Balog and Shakan in \cite{Balog}, where it is proved that
$$
| k \cdot X + l \cdot X| \geq (k+l)|X|  - (kl)^{(k+l-3)(k+l)+1}
$$
whenever $k$ and $l$ are coprime positive integers. Another preprint by Shakan \cite{Shak14} gives, in particular, an extension of
this result (where the `error term' on the right-hand side is a constant) to the case studied by Bukh of a general sum of dilates of the form $k_1 \cdot X + \cdots + k_n \cdot X$
for which the coefficients $k_i$ are positive integers and for $2 \le j \le n$ there exists an index $1 \le i < j$ such that $\gcd(k_i, k_j) = 1$.

On another hand, as appears from the above historical overview (which is almost exhaustive), very little has been done so far with regard to the problem of providing non-trivial estimates
for sums of dilates in groups different from $(\mathbb{Z}, +)$,
especially if non-abelian. A couple of exceptions are two recent papers focused on the case of cyclic groups of prime order, namely \cite{P08} and \cite{Ponti13},
some results by Konyagin and \L{}aba (see Section 3 of \cite{Kony06}) on sets of small doubling in linear spaces over the real or the rational field, and a 2014 preprint by Balog and Shakan
\cite{Balog2} on sums of dilates in $(\mathbb Z^n, +)$.

The present article fits in this context and investigates some questions related to the extension of the inequality \eqref{th:hamidoune_and_plagne_estimate_a041} to the setting of torsion-free or linearly orderable groups, either abelian or not.

\section{Going beyond the integers}
\label{sec:beyond_the_integers}

We say that an (additively written) group $G$ is \textit{linearly orderable}
if there exists a total order $\preceq$ on $G$ such that
$x+y \prec x +z$ and $y +x \prec z+x$ for all $x,y,z \in G$ with $y \prec z$, in which case we refer to $\preceq$ as a linear order on $G$ and to the pair $(G, \preceq)$ as a linearly ordered group.

Linearly orderable groups form a natural class of groups for the type of problems studied in this article. E.g., they were considered by Freiman, Herzog, Longobardi and Maj in \cite{Freim12} in reference to an extension of Freiman's $3k-3$ theorem, and some of their results were subsequently generalized
by the second-named author to the case of linearly orderable semigroups \cite{Tring13}.

In continuation to \cite{Freim12}, Freiman, Herzog, Longobardi, Maj, Stanchescu and the first-named author have recently obtained, see \cite{Freim15}, several new results on the structure of the subgroup generated by a small doubling subset of an ordered group (abelian or not), together with various generalizations of Freiman's $3k-3$ and $3k-2$ theorems.

The class of linearly orderable groups is closed under embeddings and direct products,
and it notably includes torsion-free nilpotent groups
(and so, in particular, abelian torsion-free groups), as established by Iwasawa \cite{Iwasa48}, Malcev \cite{Malc48}
and Neumann \cite{Neum49}; pure braid groups \cite{RZ98}; and free groups \cite{Iwasa48}.

In the present context, linearly orderable groups serve as a basic generalization of the additive group of the integers, in that \eqref{equ:basic} and its proof (which is essentially based on the order structure of $\Z$)
can be immediately transposed. Indeed, if $G$ is a linearly orderable group and $X$ and $Y$ are non-empty subsets of $G$, then
\begin{equation}
\label{prop:folklore_linear_a041}
|X + Y | \ge |X| +  |Y| - 1.
\end{equation}
In fact, we know even more since Kemperman proved \cite{Kemp} that \eqref{prop:folklore_linear_a041} remains true if
the ambient group is only supposed to be torsion-free.
In particular, if $k$ and $l$ are non-zero integers and $X$ is a non-empty subset of a torsion-free (and in particular linearly orderable)
group, then
\begin{equation}
\label{prop:folklore46}
|k\cdot X+l\cdot X| \ge 2|X| - 1.
\end{equation}

Based on this, we have the following result, which counts as an elementary generalization of  \eqref{th:hamidoune_and_plagne_estimate_a041}
and serves as a starting point for the present work.

\begin{theorem}
\label{th:commutative_case_a041}
Let $G$ be a torsion-free abelian group and $X$ be a non-empty finite subset of $G$.
Let $k$ and $l$ be non-zero integers with distinct absolute values.
Then $|k\cdot X+l\cdot X| \ge 3|X|-2$.
\end{theorem}

\begin{proof}
Without loss of generality we assume that $|l| > |k|\geq 1$, in particular $|l| \geq 2$. In the case when $|X|=1$, the result is immediate, so let $|X| \geq 2$ and define
$$
k^\prime = \frac{k}{\gcd(k,l)}\qquad \text{and}\qquad l^\prime = \frac{l}{\gcd(k,l)}.
$$
Clearly, $k^\prime$ and $l^\prime$ are non-zero integers with distinct absolute values
and, $G$ being abelian, we have
$$
|k \cdot X + l \cdot X| = |k^\prime \cdot (X - g) + l^\prime \cdot (X - g)|
$$
for every $g \in G$. Therefore, we can assume without loss of generality
that $0$ belongs to $X$, and $k$ and $l$ are coprime.

Furthermore, since $X$ is finite and $G$ is abelian and torsion-free, the subgroup of $G$ generated by $X$ is isomorphic to some finite power of $(\Z, +)$.
Hence, we can as well suppose that $X$ generates $G = (\mathbb{Z}^n, +)$ for some integer $n \ge 1$.
We may thus define the greatest common divisor of $X$, which is by definition the greatest common divisor of the coordinates of the elements of $X \subseteq \Z^n$.
If $d ={\rm gcd }\ X$, we obtain $X \subseteq (d \cdot \Z)^n$. This implies $\Z^n = \langle X \rangle \subseteq (d \cdot \Z)^n$ and thus $d=1$.

Let now $H$ be the subgroup of $G$ generated by $l \cdot X$.  One has
$H= \langle l \cdot X \rangle  \subseteq (l \cdot \Z)^n$. We may decompose $k\cdot X$ into its non-empty intersections with cosets modulo $H$.
Let $g_1, \ldots, g_q \in G$ be such that the cosets $g_1 + H, \ldots, g_q + H$ are pairwise disjoint and
the intersection $K_i$ of $k \cdot X$ with $g_i + H$ is non-empty. We obtain a partition of
$k \cdot X$ in the form
$$
k \cdot X = \bigcup_{i=1}^q K_i.
$$
From this, we infer
\begin{equation}
\label{equ:disjoint_sums}
|k \cdot X + l \cdot X| = \sum_{i=1}^q |K_i + l \cdot X|.
\end{equation}
On another hand, since $G$ is supposed to be equal to $(\mathbb{Z}^n, +)$,
it is linearly orderable (e.g., by the natural lexicographic order induced by the usual order of the integers). So, we get from
\eqref{prop:folklore46} and \eqref{equ:disjoint_sums} that
\begin{equation}
\label{blabla}
|k \cdot X + l \cdot X| \ge \sum_{i=1}^q (|K_i| + |X| - 1) = (q+1)|X| - q.
\end{equation}
Suppose that $q=1$. Since  $0 \in X$, we would obtain $k \cdot X \subseteq H \subseteq (l \cdot \Z)^n$. However, since $\gcd(k,l) = 1$, this would imply
that $X$ itself is included in $(l \cdot \Z)^n$ which in turn gives that $l$ divides $\gcd X =1$ and finally $|l|=1$, a contradiction. Thus we must have $q \geq 2$, and
\eqref{blabla} implies the claim.
\end{proof}

It is natural to ask if Theorem \ref{th:commutative_case_a041} continues to hold even without the assumption that the ambient group is abelian, and
our next theorem shows that the answer to this question is negative. Note that from this point on, unless a statement to the contrary is made, we
write groups multiplicatively whenever they are non-abelian.

\begin{theorem}
\label{rty}
Let $k$ and $l$ be non-zero integers. There exist a torsion-free group $G$ and a $2$-element subset
$X$ of $G$ such that $|X^{\cdot k} X^{\cdot l}| < 3|X|-2$.
\end{theorem}

Later on, we will discuss the extension of this result to a $3$-element larger $X$, but it seems extremely difficult to understand if this can be pushed further.

Nonabelian torsion-free groups are generally believed
to enjoy better isoperimetric inequalities than the abelian ones. There
is, for instance, a conjecture by Freiman stating that the inequality
$|2X| < 3|X|-4$ for a subset in a torsion-free group implies that the subgroup
generated by $X$ is cyclic (and then the structure of $X$ can be recovered by
Freiman's $3k-4$ theorem). To our knowledge, the best general
inequality for torsion-free groups can be found in \cite{BPS}, where the main result implies that, for
any constant $k$ one has $|2X| \ge 2|X|+k$ provided that $|X|$ is large enough
and not contained in a coset of a cyclic subgroup.

This may indicate that
small sets are precisely not the ones where good estimates will be found.
However, an explicit example is given in \cite{BPS} where
$$|X^2| = \frac{10}{3} |X|-\frac{13}{3}, $$
which does not exclude finding arbitrarily large
sets $X$ with $|X^{\cdot k} X^{\cdot l}| < 3|X| -2$ in torsion-free groups.

\begin{proof}[Proof of Theorem \ref{rty}]
If $k = -l$, then the claim is straightforward by taking $G$ equal to the free group on the set $\{a,b\}$ and $X = \{a,b\}$.

We therefore assume for the remainder of the proof that $k \ne -l$.
A natural place where to look for an answer is then the quotient group of the free group on $\{a,b\}$
by the normal subgroup generated by the relation $a^k b^l = b^k a^l$, namely
the two-generator one-relator group, denoted by $L_{k,l}$, with presentation
$$
 \langle a, b \mid a^k b^l = b^k a^l \rangle.
$$
In fact, by the defining relation of $L_{k,l}$ we have that taking $X = \{a,b\}$ yields
$$
|X^{\cdot k} X^{\cdot l}| = | \{ a^{k+l}, a^k b^l , b^{k+l} \} | = 3 = 3|X| - 3.
$$
Moreover, $L_{k,l}$ is torsion-free, which follows from the fact that a one-relator group has torsion if and only if
the relator is a proper power of some other element in the ambient free group, see Proposition 5.18 of \cite{LS77}.

So, we are left to check that the equation
\begin{equation}\label{equ:in_the_free_group}
a^k b^l a^{-l} b^{-k} = x^t
\end{equation}
has no solution for $t$ an integer $\ge 2$ and $x$ an element in the free group constructed on the set $\{a,b\}$.
Suppose the contrary, and let $x = x_1^{u_1} \cdots x_n^{u_n}$, where the $u_i$ are non-zero integers and the $x_i$ belong to $\{a, b\}$, with $x_i^{-1} \ne x_{i+1} \ne x_i$ for each $i = 1, \ldots, n-1$. If $y_1^{v_1} \cdots y_m^{v_m}$ is the reduced form of $x^t$, then $y_1 = x_1$ and $y_m = x_n$.
Indeed, the positive power of a reduced word, say $w$, in a free group is a word which, when reduced to its minimal form, starts and ends with the same letters as $w$.
It follows from \eqref{equ:in_the_free_group} that $x_1 = a$ and $x_n = b$, with the result that $m = nt$. Thus, we get $4 = m \ge 2t$, which is possible only if $n = t = 2$. But this yields
$$
a^k b^l a^{-l} b^{-k} = a^u b^v a^u b^v
$$
for certain non-zero integers $u$ and $v$. This implies both $u = k$ and $u = -l$, which is impossible since $k \ne -l$.
\end{proof}

As already mentioned, any linearly orderable group is torsion-free, but not viceversa, as is shown, for instance,
by the two-generator one-relator group defined by the presentation $\langle a, b \mid a^2 = b^2 \rangle$,
which is actually the fundamental group of the Klein bottle, see Example 1.24 in \cite{Ha}.

The question therefore remains whether or not it is possible to give an extension
of Theorem \ref{th:commutative_case_a041} to the case where the group $G$ is linearly orderable (and not just torsion-free).
Since we cannot decide whether or not the group $L_{k,l}$, used in the proof above, is linearly orderable, Theorem \ref{rty} cannot be immediately extended to linearly orderable groups.

However, the following theorem shows that the answer to this question is negative too, and it represents the main contribution of the paper.

\begin{theorem}
\label{maintheorem}
Let $k$ and $l$ be positive integers. There then exist a linearly orderable group $G$ and a $2$-element set $X \subseteq G$ such that $|k\cdot X+l\cdot X| < 3|X|-2$.
\end{theorem}

The proof is deferred to the next section. In Section \ref{sec4}, we make a first step in the refinement of these results by passing from a $2$-element to a $3$-element set.
At the moment, our constructions do not allow for sets of cardinality larger than $3$. This leads us, in Section \ref{sec5}, to raise some questions in this direction.

\section{Proof of Theorem \ref{maintheorem}}

We start by recalling the following result by Iwasawa (namely, Lemma 1 in \cite{Iwasa48}),
which gives a practical characterization of orderable groups and will be used below to construct the linearly orderable group (in fact, a semidirect product) required by our proof of Theorem \ref{maintheorem}.

\begin{lemma}[Iwasawa's lemma]
A group $G$
is linearly orderable if and only if there exists a subset $P$ of $G$, referred to as a positive cone of $G$, such that:
\begin{enumerate}[label={\rm (\roman{*})}]
\item\label{item:iwasawa_i} $1 \notin P$,
\item\label{item:iwasawa_ii} for each $x \in G \setminus \{1\}$ either $x \in P$ or $x^{-1} \in P$,
\item\label{item:iwasawa_iii}  for all $x,y \in P$, $xy \in P$, and
\item\label{item:iwasawa_iv} if $x \in P$ then $yxy^{-1} \in P$ for all $y \in G$.
\end{enumerate}
\end{lemma}
Given two groups $G$ and $H$ and a
homomorphism $\varphi: H \to {\sf Aut}(G)$, we denote as usual by $G \rtimes_\varphi H$ the semidirect product of $G$ by $H$
under $\varphi$, that is the (multiplicatively written) group with elements in the set $G \times H$ whose operation is defined by
$$
(g_1,h_1)
(g_2, h_2) = (g_1  \varphi_{h_1}(g_2), h_1 h_2)
$$
for all $g_1, g_2 \in G$ and $h_1, h_2 \in H$. We recall that the identity in $G \rtimes_\varphi H$ is the pair $(1,1)$, while
the inverse of an element $(g,h) \in G \rtimes_\varphi H$  is $(\varphi_{h^{-1}}(g^{-1}), h^{-1})$.

With this notation at hand, we have now the following lemma,
which provides a practical method for constructing a number of
linearly orderable groups. This is essentially a generalization of a result by Conrad (see the first few lines
of Section 3 in \cite{Con}), where $H$ is assumed to be a subgroup of ${\sf Aut}(G)$ and $\varphi$ is the restriction to $H$
of the trivial automorphism on ${\sf Aut}( G)$, that is the identity of ${\sf Aut}({\sf Aut}(G))$.

\begin{lemma}
\label{lem:a041_construction_lemma}
Let $(G, \preceq_G)$ and $(H, \preceq_H)$ be linearly ordered groups, and let $\varphi$ be a homomorphism $H \to {\sf Aut}(G)$.
If $\varphi$ is order-preserving, in the sense that $1 \prec \varphi_h(g)$ for all $g \in G$ and $h \in H$ with $1 \prec g$, then $G \rtimes_\varphi H$ is a linearly orderable group.
\end{lemma}

\begin{proof}
We define $P$ as
$$
P= \{ (g,h) \in G \times H \text{ with either } 1 \prec h , \text{ or } h = 1,\ 1 \prec g \}.
$$
We shall prove that $P$ is a positive cone of $G \rtimes_\varphi H$ by checking the conditions \ref{item:iwasawa_i}-\ref{item:iwasawa_iv} in the statement of Iwasawa's lemma.

It is clear by construction that $(1, 1) \notin P$, so \ref{item:iwasawa_i} is immediate.

As for \ref{item:iwasawa_ii}, let $(g,h) \in P$ and assume for a contradiction that its inverse $(\varphi_{h^{-1}}(g^{-1}), h^{-1})$ is also in $P$.
By the definition of $P$, $(g,h) \in P$ implies that $h$ must be larger than or equal to $1$. But our assumptions give $1 \preceq h^{-1}$, so we must have $h = 1$. It follows that $1 \prec g$ and $1 \prec \varphi_{h^{-1}}(g^{-1})=\varphi_{1}(g^{-1})=g^{-1}$, a contradiction.

Then we come to \ref{item:iwasawa_iii}. Given $(g_1, h_1),(g_2, h_2) \in P$, we compute the product  $(g_1, h_1)  (g_2, h_2) = (g_1 \varphi_{h_1}(g_2), h_1 h_2)$ and observe that this is still in $P$.
Indeed, both $h_1$ and $h_2$ are larger than or equal to $1$, so does their product. If $1 \prec h_1 h_2$, we are done. Otherwise,
we must have $h_1 =h_2 =1$. In this case, $g_1 \varphi_{h_1}(g_2) = g_1 g_2$, which is larger than $1$ since both $1 \prec g_1$ and $1 \prec g_2$, by the definition of $P$.

Finally, we prove \ref{item:iwasawa_iv}:  if $(g_1,h_1) \in G \times H$ and $(g_2,h_2) \in P$, then
\begin{equation*}
\begin{split}
(g_1,h_1)  (g_2,h_2)  (g_1,h_1)^{-1}
                & = (g_1, h_1)  (g_2 \varphi_{h_2 h_1^{-1}}(g_1^{-1}), h_2h_1^{-1})\\
                & = (g_1 \varphi_{h_1}(g_2) \varphi_{h_1 h_2 h_1^{-1}}(g_1^{-1}), h_1 h_2 h_1^{-1}).
\end{split}
\end{equation*}
If $1 \prec h_2$, then $1 \prec h_1 h_2 h_1^{-1}$ and we are done. Otherwise, $h_2 = 1$ and $1 \prec g_2$, with the result that
$$
(g_1,h_1)  (g_2,h_2)  (g_1,h_1)^{-1} = (g_1 \varphi_{h_1}(g_2) g_1^{-1}, 1) \in P,
$$
where we use that $1 \prec \varphi_{h_1}(g_2 )$ by the fact that $\varphi$ is order-preserving.

Putting all together, we then have that $P$ is a positive cone of $G \rtimes_\varphi H$, and thus the claim follows from Iwasawa's lemma.
\end{proof}

At this point, we need to introduce some more notation. Specifically,
given a subset $S$ of $ \mathbb{R}$,
we write $\mathcal{E}[S]$ for the additive subgroup of the real field
generated by numbers of the form $\alpha^{h}$ such that $\alpha \in S$ and $h \in \mathbb{Z}$. In the case when $S$ consists of a unique element, say $\alpha$, we simply write $\mathcal{E}[\alpha]$ for $\mathcal E[\{ \alpha \} ]$.
If $\alpha$ is a positive integer, we call $\mathcal{E}[\alpha]$ the set of \textit{$\alpha$-adic fractions}, in that the elements of $\mathcal{E}[\alpha]$
are explicitly given by the rational fractions of the form $h/\alpha^t$ for which $h \in \Z$ and $t \in \mathbb{N}$.

\begin{theorem}
\label{th:semidirect_products_are_lo_a041}
Let $k$ be a positive integer and $\alpha$ a positive real number. Consider the set
$$
S = \left\{\alpha^{k^{-i}} \text{ for } i = 0, 1, 2, \ldots\right\}
$$
and let $\varphi$ be the homomorphism $(\mathcal{E}[k],+) \to {\sf Aut}(\mathcal{E}[S], +)$ taking a $k$-adic fraction $u$ to the automorphism
$x \mapsto \alpha^u x$ of $(\mathcal{E}[S],+)$. Finally, let $G$ be the semidirect product
$$
G = (\mathcal{E}[S], +) \rtimes_\varphi (\mathcal{E}[k], +).
$$
The following holds:
\begin{enumerate}[label={\rm (\roman{*})}]
\item \label{item:1_th1_a041} $G$ is a linearly orderable group.
\item \label{item:3_th1_a041} Given an integer $l  \geq k$, there is a real number $\beta_{k,l} > 0$ satisfying the equation
$$
\sum_{i=k+1}^{k+l} x^{i/k} -\sum_{i=1}^{k} x^{i/k} =0.
$$
Then, for $\alpha= \beta_{k,l}$ we have
$$
e_0^k e_1^l = e_1^k e_0^l
$$
where $e_0 = (\alpha^{1/k}, 1/k)$ and $ e_1 = (0,1/k)$
are elements of $G$.
\end{enumerate}
\end{theorem}

\begin{proof}
Let us first check that $\varphi_u$ is well-defined as a function $\mathcal{E}[S] \to \mathcal{E}[S]$ for each $u \in \mathcal{E}[k]$.
For, $u$ being a $k$-adic fraction means that there exist $h \in \mathbb{Z}$ and $t \in \mathbb{N}$ such that $u = h/k^t$.
Therefore, using that an element of $\mathcal E[S]$ is a real number of the form $\sum_{i=1}^l g_i \alpha^{h_i/k^{t_i}}$ for certain $g_i, h_i \in \mathbb{Z}$ and $t_i \in \N$, we find that
$$
\varphi_u(x) = \sum_{i=1}^l g_i \alpha^{h/k^{t} + h_i/k^{t_i}}  = \sum_{i=1}^l g_i \alpha^{s_i/k^{r_i}} \in \mathcal{E}[S],
$$
where $r_i = \max(t_i, t)$ and $s_i$ is equal to the integer $k^{r_i-t} h + k^{r_i - t_i} h_i$.
It follows that $\varphi$ is also well-defined as a homomorphism $(\mathcal{E}[k], +) \to {\sf Aut}(\mathcal{E}[S], +)$.
In fact, $\varphi_{u+v} = \varphi_u \circ \varphi_v$ for all $u,v \in \mathcal{E}[k]$, and $\varphi_u$ is bijective, its inverse being $\varphi_{-u}$.
Lastly, for $u \in \mathcal{E}[k]$ and $x,y \in \mathcal{E}[S]$ we have $\varphi_u(x+y) = \varphi_u(x) + \varphi_u(x)$, that is $\varphi_u$ is a homomorphism
of $(\mathcal E[S],+)$.

\ref{item:1_th1_a041} The claim follows at once from Lemma \ref{lem:a041_construction_lemma}, considering that $(\mathcal{E}[S], +)$ can be linearly ordered by the usual order of $\mathbb{R}$ and  $\alpha > 0$ implies that $\varphi_u(x) = \alpha^{u} x > 0$ for each $u \in \mathcal{E}[k]$  and each positive $x \in \mathcal{E}[S]$.

\ref{item:3_th1_a041} Since the polynomial
$$
P(x)=\sum_{i=k+1}^{k+l} x^{i} -\sum_{i=1}^{k} x^{i}
$$
satisfies $P(0)=0$, $P'(0)=-1$ and $P(x) \to + \infty$ as $x \to +\infty$, it must have one or more positive roots. Let $\beta_{k,l}$ be the $k$-th power of one of them, chosen arbitrarily.

A straightforward induction shows that for any positive integer $j$, we have
$$
e_0^j = (\alpha^{1/k} + \cdots + \alpha^{j/k}, j/k) \text{ and }\ e_1^j =(0, j/k),
$$
whence we get, on the one hand,
$$
e_0^k e_1^l = (\alpha^{1/k} + \cdots + \alpha^{k/k}, 1)  (0,l/k) = (\alpha^{1/k} + \cdots + \alpha^{k/k}, 1+l/k),
$$
and on the other hand,
$$
e_1^k e_0^l = (0,1) (\alpha^{1/k} + \cdots + \alpha^{l/k}, l/k) = (\alpha^{1+1/k} + \cdots + \alpha^{1+l/k}, 1+l/k).
$$
Thus, $e_0^k e_1^l = e_1^k e_0^l$ since $\alpha=\beta_{k,l}$ satisfies $P(\alpha)=0$.
\end{proof}

Taking $X = \{e_0, e_1\}$ in Theorem \ref{th:semidirect_products_are_lo_a041}, we obtain a $2$-element set
such that
$$
|X^{\cdot k} X^{\cdot l}| = 3 = 3|X|-3.
$$
Noticing that we can always reverse the roles of $k$ and $l$, we thus have the following:

\begin{corollary}
\label{cor:1_a041}
Let $k$ and $l$ be positive integers.
There exist a linearly orderable group $G$ and a set $X \subseteq G$ of cardinality $2$ such that $|X^{\cdot k} X^{\cdot l}| = 3 = 3|X|-3$.
\end{corollary}

This result is enough to conclude our proof of Theorem \ref{maintheorem}.

\section{Extending Theorem \ref{maintheorem}}
\label{sec4}

Based on Theorem \ref{maintheorem}
it is somewhat natural to ask whether it is possible to construct linearly orderable groups where
to find sets $X$ of any possible size such that $|k \cdot X + l \cdot X| < 3|X| - 2$ for some non-zero integers $k$ and $l$.
In our final result, we show that this is actually the case with a $3$-element set.

\begin{theorem}
\label{th:generalization_to_direct_products_a041}
Let $k$ and $r$ be positive integers, and let $\alpha$ be a positive real number.
Consider the set
$$
S =\left\{\alpha^{k^{-i}} \text{ for } i = 0, 1, 2, \ldots\right\}.
$$
Define
$H$ to be the direct product of $r$ copies of $(\mathcal{E}[S], +)$, and let
$$
G = H \rtimes_\varphi (\mathcal{E}[k], +),
$$
where $\varphi$ is the homomorphism $(\mathcal{E}[k], +) \to {\sf Aut}(H)$ sending a $k$-adic fraction $u$ to the automorphism
$(x_1,\ldots, x_r) \mapsto (\alpha^u x_1, \ldots, \alpha^u x_r)$ of $H$. The following holds:
\begin{enumerate}[label={\rm (\roman{*})}]
\item \label{item:th5_i}  $G$ is a linearly orderable group.
\item \label{item:th5_ii} Let $e_i$ denote, for $i = 0, \ldots, r$, the $(r+1)$-tuple of $G$ whose $(r+1)$-th component is $1/k$, whose $i$-th component is $\alpha^{1/k}$ if $i \ne 0$, and all of whose other components are zero. Moreover, let $l$ be an integer $\ge k$ and assume $\alpha=\beta_{k,l}$, where $\beta_{k,l}$ is defined as in Theorem \ref{th:semidirect_products_are_lo_a041}, point \ref{item:3_th1_a041}.
Then,
$$
e_i^k e_j^l = e_j^k e_i^l
$$
for all $ i,j = 0, 1, \ldots, r$, and hence
$$
|X^{\cdot k} X^{\cdot l}| = 6 = 3|X|-3
$$
for $r \ge 2$ and $X = \{e_0, e_1, e_2\}$.
\end{enumerate}
\end{theorem}

\begin{proof}
We should first check that $\varphi$ is well-defined as a homomorphism from $(\E[k], +)$ to ${\sf Aut}(H)$, but this boils down to the same kind of verification as in the proof of Theorem \ref{th:semidirect_products_are_lo_a041}, so we can move on.

\ref{item:th5_i} The group $H$ can be linearly ordered by the natural lexicographic order on its $r$ components, and accordingly $\varphi$ is order-preserving. Thus $G$ is a linearly orderable group by Lemma \ref{lem:a041_construction_lemma}.

\ref{item:th5_ii} Denote by $\psi$ the homomorphism $(\E[k], +) \to {\sf Aut}(\E[S],+)$ taking an $k$-adic fraction $u$ to the automorphism $x \mapsto \alpha^u x$ of $(\E[S], +)$. Then for each $h = 1, \ldots, r$ let $\pi_h$ be the projection homomorphism $$
G \to (\E[S], +) \times_\psi (\E[k], +): (x_1, \ldots, x_r, u) \to (x_h, u).
$$
This is clearly a surjective homomorphism, and its restriction to the subgroup of $G$ consisting of those $(r+1)$-tuples $(x_1, \ldots, x_r, u)$ such that $x_j \ne 0$ for some $1 \le j \le r$ only if $j = h$ is an isomorphism.  Therefore, we find that two elements $\xi$ and $\zeta$ of $G$ are equal if and only if
$$
 \pi_h(\xi) = \pi_h(\zeta) \text{ for all } h = 1, \ldots, r.
$$
With this in hand, fix $i,j = 0, 1, \ldots, r$. We have to prove that
$e_i^l e_j^k = e_j^l e_i^k$. For, it follows from the above that this is equivalent to
$$
\pi_h(e_i)^l \pi_h(e_j)^k = \pi_h(e_j)^l \pi_h(e_i)^k
$$
for all $h = 1, \ldots, r$ (here we use that $\pi_h$ is a homomorphism), which in turn is immediate by Theorem \ref{th:semidirect_products_are_lo_a041}, since $\pi_h(e_i) = \pi_h(e_j) = (\alpha^{1/k}, 1/k)$ for $i,j \ge 1$ and $\pi_h(e_0) = (0, 1/k)$.
\end{proof}

By considering the set $X=\{ e_0, e_1,\dots, e_{r-1}\}$ in Theorem \ref{th:generalization_to_direct_products_a041}, we obtain the following corollary.

\begin{corollary}
\label{cor:2_a041}
Let $k$, $l$ and $r$ be positive integers. There exist a linearly orderable group $G$ and a set $X \subseteq G$ with $|X|=r$ such that $|X^{\cdot k} X^{\cdot l}| = \binom{r+1}{2}$.
\end{corollary}

The special case where $r=3$ is of particular interest.

\begin{corollary}
\label{kljl}
Let $k$ and $l$ be positive integers. There exist a linearly orderable group $G$ and a set $X \subseteq G$ with $|X|=3$ such that $|X^{\cdot k} X^{\cdot l}| = 6 = 3|X|-3$.
\end{corollary}

\section{Some questions for further research}
\label{sec5}

For a group $G$ and non-zero integers $k$,  $l$ and $r \ge 1$, let us denote by
$$
\chi_{G}(k,l,r)
$$
the minimal possible cardinality of a sumset of the form $X^{\cdot k} X^{\cdot l}$ for a set $X \subseteq G$ with $|X|=r$.
Immediate bounds for this function are
$$
r \leq \chi_{G}(k,l,r) \leq r^2.
$$
The lower inequality is in general sharp: it is enough to consider groups containing a cyclic subgroup of order $r$. However, it can be improved in a number of cases
by using, for instance, a Kneser type or Cauchy-Davenport type theorem.

Based on the above, it is then natural to define, for $\G$ a given class of groups,
$$
\chi^{\G}(k,l,r) = \inf_{G\in \G}\chi_{G} (k,l,r).
$$

In this paper, we have especially investigated the case when $\G$ is either the class $\mathsf{TF}$ of torsion-free groups, or the class $\mathsf{LO}$ of linearly orderable groups.
Using $\mathsf{LO} \subsetneq \mathsf{TF}$, \eqref{prop:folklore46} and the fact that $(\mathbb{Z}, +)$ is an element of $\mathsf{LO}$  yields
\begin{equation}
\label{eq:trivial_on_LO_chi}
2r-1 \leq \chi^{\sf TF} (k,l,r) \leq \chi^{\mathsf{LO}}(k,l,r)  \leq \chi_{(\mathbb{Z}, +)}(k,l,r) \leq (|k|+|l|)r-(|k|+|l|-1).
\end{equation}
The last upper bound follows from considering $X=\{0,1,\dots,r-1\}$.

If $|k|=|l|=1$, \eqref{eq:trivial_on_LO_chi} is in fact a series of equalities.
In any other case, we must have $|k|+|l| \geq 3$ and the upper bound we obtain from \eqref{eq:trivial_on_LO_chi} is never better than $3r-2$.
This bound is improved by Theorems \ref{th:semidirect_products_are_lo_a041} and \ref{th:generalization_to_direct_products_a041}
since these results, when combined with equation \eqref{eq:trivial_on_LO_chi}, read as
$$
\chi^{\sf TF} (k,l,2) =\chi^{\sf LO} (k,l,2)=3\quad\text{and}\quad 5 \leq \chi^{\sf TF} (k,l,3) \leq \chi^{\sf LO} (k,l,3)\leq 6
$$
for any positive integers $k$ and $l$.
Corollary \ref{kljl} provides a quadratic bound in $r$ independent from $k$ and $l$, namely $r(r+1)/2$. But unfortunately, this is smaller than the linear upper bound in  \eqref{eq:trivial_on_LO_chi}
only for small values of $r$. In particular, we get a better bound than $3r-2$ only if $r=2$ or $3$.

Although the situation is clear for $r=2$, even the question of which is the exact value of $\chi^{\sf LO} (k,l,3) \in \{5,6\}$ remains open.
It is not difficult to see that the answer is $5$ if and only if there exist a linearly orderable group $G$ and elements $x,y,z \in G$
such that (i) $x \prec y \prec z$, (ii) $x^ky^l = y^kx^l$, (iii) $x^kz^l=z^kx^l=y^{k+l}$, and (iv) $y^k z^l = z^k y^l$,
but this looks challenging to investigate even in the basic case, say, when $k= 1$ and $l = 2$.

A related question  is as follows: Let $L^{(r)}$ be the free group on $r$ variables, say $x_1,\dots, x_r$.
We denote by $L^{(r)}_{k,l}$ the quotient group of $L^{(r)}$
by the normal subgroup generated by the relations $x_i^k x_j^l = x_j^k x_i^l$ as $i$ and $j$ range in the interval $\{1,\dots, r\}$.
When $r=2$, this is the two-generator one-relator
group considered at the end of Section \ref{sec:beyond_the_integers}.
Are the $L^{(r)}_{k,l}$  linearly orderable groups? Are they torsion-free?
If $r=2$, one can see using von Dyck's theorem (namely, Theorem 2.2.1 in \cite{Rob81})
that there exists an epimorphism $\vartheta: L^{(2)}_{k,l} \to G$, the group constructed in Theorem \ref{th:semidirect_products_are_lo_a041},
mapping $x_1$ to $(\alpha, 1)$ and $x_2$ to $(0,1)$.
We ask if $\vartheta$ is actually an isomorphism. If the answer to this question were yes, this would imply that $L^{(2)}_{k,l}$ is linearly orderable.

More generally, let ${\mathscr R}$ be a set of independent relations on the $r$ variables of the free group $L^{(r)}$, each being of the form
$$
x_i^k x_j^l = x_u^k x_v^l
$$
where $i,j,u$ and $v$ belong to $\{1,\dots, r\}$.
We denote by $L^{(r)}_{{\mathscr R}}$ the quotient group of $L^{(r)}$
by the normal subgroup generated by the set of relations $\R$.
How large can $\R$ be if we require that $L^{(r)}_{{\mathscr R}}$ is linearly orderable or torsion-free?
Since each of the above relations makes the cardinality of the sum $X^{\cdot k} X^{\cdot l}$, where $X$ is the set $\{x_1,\dots,x_r\}$, decrease by one, looking for a large $\R$ is tightly related to having good upper bounds on  $\chi^{\sf TF}(k,l,r)$ and  $\chi^{\mathsf{LO}}(k,l,r)$ and could help to understand the behaviour of these
functions.

\section*{Acknowledgements}

The second-named author is grateful to Yves  de Cornulier for a useful observation
which eventually led to the formulation of Lemma \ref{lem:a041_construction_lemma}.


\begin{thebibliography}{99}


\bibitem{Balog} A. Balog, and G. Shakan, \textit{On the sum of dilations of a set}, Acta Arith. \textbf{164} (2014), 153--162.

\bibitem{Balog2} \bysame, \textit{Sum of dilates in vector spaces}, preprint (\href{http://arxiv.org/abs/1410.8614}{arXiv:1410.8614}), last updated: Oct 31, 2014.

\bibitem{BPS} K. J. B\"or\"oczky, P. P. P\'alfy, and O. Serra, \textit{On the cardinality of sumsets in torsion-free groups}, Bull. Lond. Math. Soc. \textbf{44}, No. 5 (2012), 1034--1041.

\bibitem{BS} G. Baumslag and D. Solitar, \textit{Some two-generator one-relator non-Hopfian groups}, Bull. Amer. Math. Soc. \textbf{68} (1962), 199--201.

\bibitem{BB} B. Bukh, \textit{Sums of dilates}, Combin. Probab. Comput. \textbf{17}, No. 5 (2008), 627--639.

\bibitem{CHS} J. Cilleruelo, Y. O. Hamidoune, and O. Serra, \textit{On sums of dilates}, Combin. Probab. Comput.  \textbf{18}, No. 6 (2009), 871--880.

\bibitem{CSV} J. Cilleruelo, M. Silva, and C. Vinuesa, \textit{A sumset problem}, J. Comb. Number Theory \textbf{2}, No. 1 (2010), 79--89.

\bibitem{Con} P. Conrad, \textit{Non-abelian ordered groups}, Pacific J. Math. \textbf{9}, No. 1 (1959), 25--41.

\bibitem{Deho02} P. Dehornoy, I. Dynnikov, D. Rolfsen, and B. Wiest, \textit{Why are braids orderable?}, Panoramas \& Synth\`eses \textbf{14}, Soc. Math. France, 2002.

\bibitem{Freim12} G. A. Freiman, M. Herzog, P. Longobardi, and M. Maj, \textit{Small doubling in ordered groups}, J. Austral. Math. Soc. \textbf{96}, No. 3 (2014), 316--325.

\bibitem{Freim13} G.A. Freiman, M. Herzog, P.~Longobardi, M.~Maj, and Y.~V.~Stanchescu, \textit{Inverse problems in additive number theory and in non-abelian group theory}, European J. Combin. \textbf{40} (2014), 42--54.

\bibitem{Freim15} G.A. Freiman, M. Herzog, P.~Longobardi, M.~Maj, A. Plagne, and Y.~V.~Stanchescu, \textit{Small doubling in ordered groups: generators and structure}, preprint (\href{http://arxiv.org/abs/1501.01838}{arXiv:1501.01838}), last updated: Jan 8, 2015.

\bibitem{Gara07} M. Z. Garaev, \textit{An explicit sum-product estimate in $\mathbb{F}_p$}, Int. Math. Res. Not. \textbf{11} (2007), Art. ID rnm035.


\bibitem{Hami00} Y. O. Hamidoune and A. Plagne, \textit{A generalization of Freiman's $3k-3$ theorem}, Acta Arith. \textbf{103}, No. 2 (2002), 147--156.

\bibitem{HR}  Y. O. Hamidoune and J. Ru\'e, \textit{A lower bound for the size of a Minkowski sum of dilates}, Combin. Probab. Comput. \textbf{20}, No. 2 (2011), 249--256.

\bibitem{Ha} A. Hatcher, \textit{Algebraic topology}, Cambridge University Press, 2002.

\bibitem{Iwasa48} K.~Iwasawa, \textit{On linearly ordered groups}, J. Math. Soc. Japan \textbf{1} (1948), 1--9.

\bibitem{Kemp} J. H. B.~Kemperman, \textit{On complexes in a semigroup}, Indag. Math. \textbf{18} (1956), 247--254.

\bibitem{Kony06} S. Konyagin and I. \L{}aba, \textit{Distance sets of well-distributed planar sets for polygonal norms}, Israel J. Math \textbf{152} (2006), 157--179.

\bibitem{LS77} R. C. Lyndon and P. E. Schupp, \textit{Combinatorial Group Theory}, Springer-Verlag, 1977.

\bibitem{Malc48} A.~I.~Malcev, \textit{On ordered groups}, Izv. Akad. Nauk. SSSR Ser. Mat. \textbf{13} (1948), 473--482.

\bibitem{Na} M. B. Nathanson, \textit{Additive Number Theory: Inverse Problems and the Geometry of Sumsets}, Vol. 165 in the Graduate Texts in Mathematics series, Springer, 1996.

\bibitem{N} M. B. Nathanson, \textit{Inverse problems for linear forms over finite sets of integers}, J. Ramanujan Math. Soc. \textbf{23}, No. 2 (2008), 151--165.

\bibitem{Neum49} B.~H.~Neumann, \textit{On ordered groups}, Amer. J. Math. \textbf{71} (1949), 1--18.

\bibitem{P08} A. Plagne, \textit{Sums of dilates in groups of prime order}, Combin. Probab. Comput. \textbf{20}, No. 6 (2011), 867--873.

\bibitem{Ponti13} G. F. Pontiveros, \textit{Sums of Dilates in $\mathbb{Z}_p$}, Combin. Probab. Comput. \textbf{22}, No. 2 (2013), 282--293.

\bibitem{Rob81} D. J. S. Robinson, \textit{A Course in the Theory of Groups}, Springer-Verlag, 1982.

\bibitem{RZ98} D. Rolfsen and J. Zhu, \textit{Braids, orderings and zero divisors}, J. Knot Theory Ramifications \textbf{7}, No. 6 (1998), 837--841.

\bibitem{Shak14} G. Shakan, \textit{A bound for the size of the sum of dilates}, preprint (\href{http://arxiv.org/abs/1402.4721}{arXiv:1402.4721}), last updated: Jan 14, 2014.

\bibitem{TrMTK} S. Tringali, \textit{Cauchy-Davenport type theorems for semigroups},
    to appear in Mathematika (\href{http://arxiv.org/abs/1307.8396}{arXiv:1307.8396}).

\bibitem{Tring13} \bysame, \textit{Small doubling in ordered semigroups},
    Semigroup Forum \textbf{90}, No. 1 (2015), 135--148.

\end{thebibliography}
\end{document}